\documentclass[10pt]{amsart}

\usepackage{amsmath, amscd, amssymb}
\usepackage[frame,cmtip,arrow,matrix,line,graph,curve]{xy}
\usepackage{graphpap, color}
\usepackage[mathscr]{eucal}
\usepackage{cancel}
\usepackage{verbatim}

\usepackage{pgfplots} 

\numberwithin{equation}{section}

\def\sO{{\mathscr O}}

\newcommand{\A}{\mathbb{A}}

\newcommand{\GG}{\mathbb{G}}
\newcommand{\KK}{\mathbb{K}}
\newcommand{\LL}{\mathbb{L}}
\newcommand{\NN}{\mathbb{N}}
\newcommand{\PP}{\mathbb{P}}
\newcommand{\QQ}{\mathbb{Q}}

\newcommand{\Ss}{\mathbb{S}}
\newcommand{\ZZ}{\mathbb{Z}}

\newcommand{\bVar}{\mathbf{Var}}
\newcommand{\Sch}{\mathbf{Sch}}

\newcommand{\SmProjVar}{\mathbf{SmProjVar}}

\newcommand{\cal}{\mathcal}

\def\cC{{\cal C}}

\def\cE{{\cal E}}

\def\cO{{\cal O}}

\def\cZ{{\cal Z}}

\def\cX{\mathcal{X} }

\def\and{\quad{\rm and}\quad}

 \DeclareMathOperator{\Ext}{Ext}
  \DeclareMathOperator{\Hom}{Hom}

\newtheorem{prop}{Proposition}[section]
\newtheorem{theo}[prop]{Theorem}

\newtheorem{coro}[prop]{Corollary}
\newtheorem{rema}[prop]{Remark}

\newtheorem{defi}[prop]{Definition}
\newtheorem{conj}[prop]{Conjecture}

\def\beq{\begin{equation}}
\def\eeq{\end{equation}}

\def\PP{\mathbb{P}}

\def\cO{\mathcal{O}}

\def\Chow{\mathrm{Chow}}
\def\DM{\mathrm{DM}}
\def\lim{\mathrm{lim}}

\def\dim{\mathrm{dim}}
\def\Bun{\mathrm{Bun}}

\def\Pic{\mathrm{Pic}}

\def\eff{\mathrm{eff}}
\def\gm{\mathrm{gm}}

\title[Motivic decompositions of moduli of vector bundles on curves]{Motivic decompositions of moduli spaces of vector bundles on curves}

\author[T. G\'omez]{Tom\'as L. G\'omez}
\address{Instituto de Ciencias Matem\'aticas (CSIC-UAM-UC3M-UCM), Nicol\'as Cabrera 15, Campus Cantoblanco UAM, 28049 Madrid, Spain}
\email{tomas.gomez@icmat.es}

\author[K.-S. Lee]{Kyoung-Seog Lee}
\address{University of Miami, Department of Mathematics, 1365 Memorial Drive, Ungar 515, Coral Gables, FL 33146} 
\email{kyoungseog02@gmail.com}



\begin{document}

\begin{abstract}
Let $r \geq 2, d$ be two integers which are coprime to each other. Let $C$ be a smooth projective curve of genus $g \geq 2$ and $M(r,L)$ be the moduli space of rank $r$ stable vector bundles on $C$ whose determinants are isomorphic to a fixed line bundle $L$ of degree $d$ on $C.$ In this paper, we study motivic decomposition of $M(r,L)$ for $r=2, 3$ cases. We give a new proof of a version of the main result of \cite{Lee}. We also found a new motivic decomposition of $M(3,L).$
\end{abstract}
\maketitle

\section{Introduction}

Let $\KK$ be an algebraically closed field whose characteristic is 0. Let $r \geq 2, d$ be two integers which are coprime to each other. Let $C$ be a smooth projective curve of genus $g \geq 2$ over $\KK$ and $M(r,L)$ be the moduli space of rank $r$ stable vector bundles on $C$ whose determinants are isomorphic to a fixed line bundle $L$ of degree $d$ on $C.$ It turns out that $M(r,L)$ is a beautiful algebraic variety which reflects many properties of $C$ and there have been intensive works studying various aspects of $M(r,L).$ For examples, cohomology groups and motives of $M(r,L$) are known to be closely related to those invariants of $C.$ Recently bounded derived category of coherent sheaves on $M(r,L)$ also draws lots of attentions. Because $M(r,L)$ is a Fano variety, its derived category determines the variety and has a nontrivial semiorthogonal decomposition. Moreover, there is a natural functor from the derived category of $C$ to the derived category of $M(r,L).$ Let $D(C)$ (resp. $D(M(r,L))$) be the bounded derived category of coherent sheaves on $C$ (resp. $M(r,L)$). It is well-known that there exists a Poincar\'{e} bundle $\cE$ on $C \times M.$ Then the Poincar\'{e} bundle $\cE$ gives a Fourier-Mukai transform
$$ \Phi_{\cE} : D(C) \to D(M(r,L)) $$
defined by $\Phi_{\cE}(F) := {Rp_{M(r,L)}}_*({Lp_C}^*(F) \otimes^L \cE)$, where $F$ is an element in $D(C)$ and $p_C$ (resp. $p_{M(r.L)}$) is the projection map from $C \times M(r,L)$ to $C$ (resp. $M(r,L)$). It was proved that $\Phi_{\cE}$ is fully faithful for every smooth projective curve of genus greater than or equal to 2 in \cite{FK, Narasimhan1, Narasimhan2} when $r=2.$ It was also proved that $\Phi_{\cE}$ is fully faithful for every smooth projective curve of sufficiently large genus in \cite{BS} when $r \geq 3, d=1.$

\bigskip

It is a natural and important task to understand the full semiorthogonal decomposition of $D(M(r,L)).$ M. S. Narasimhan conjectured that $D(M(2,L))$ will have a semiorthogonal decomposition consisting of two copies of the derived category of a point, two copies of the derived category of $C,$ $\cdots,$ two copies of the derived category of $C_{g-2}$ and one copy of the derived category of $C_{g-1},$ where $C_{k}$ denotes the $k$-th symmetric power of $C.$ We were informed that Belmans, Galkin and Mukhopadhyay stated the same conjecture in \cite{BGM} independently.

\begin{conj}\label{SODrank2}
The derived category of $M(2,L)$ has the following semiorthogonal decomposition
$$ D(M(2,L))=\langle D(pt), D(pt), D(C), D(C), \cdots, D(C_{g-2}), D(C_{g-2}), D(C_{g-1}) \rangle. $$
\end{conj}

On the other hand, Orlov predicted that derived categories of coherent sheaves and motives of algebraic varieties will be closely related in \cite{Orlov}. When $r=2, d=1,$ it turns out that the motive of $M(2,L)$ in any semisimple category of motives has motivic decomposition which is compatible with the above conjecture (cf. \cite{Lee}). The motivic decomposition obtained in \cite{Lee} has the following consequence.

\begin{theo}\cite{Lee}\label{rank2}
Let $r=2, d=1.$ The motivic Poincar\'e polynomial $\chi(M(2,L))$ is
$$ \sum_{k=0}^{g-2} \chi(C_{k})  (\mathbb{L}^{ k} + \mathbb{L}^{ 3g-3-2k}) + \chi(C_{g-1})  \mathbb{L}^{ g-1}. $$
\end{theo}

It is natural to attempt to find similar formulas for other rank and degree cases. In this paper we provide new decompositions of motivic Poincar\'e polynomials of moduli spaces of rank three stable vector bundles on curves.

\begin{theo}\label{rank3}
Let $r=3, d=1.$ The motivic Poincar\'e polynomial $\chi(M(3,L))$ is
$$ \sum_{k_1+k_2 < 2(g-1)} \chi(C_{k_1} \times C_{k_2})  (\mathbb{L}^{ k_1+2k_2} + \mathbb{L}^{ 8g-8-2k_1-3k_2}) + $$
$$ \sum_{k_1+k_2=2(g-1), k_1<g-1} \chi(C_{k_1} \times C_{k_2})  (\mathbb{L}^{ k_1+2k_2} + \mathbb{L}^{ 8g-8-2k_1-3k_2}) + \chi(C_{g-1} \times C_{g-1})  \LL^{ 3(g-1)}. $$
\end{theo}

A remarkable property of these formulas is that they only involve sums of products of symmetric products of $C$. \\

Using the same strategy we can also obtain the same motivic decompositions of the moduli space in the dimensional completions of Grothendieck ring of varieties or Voevodsky motives. See Theorem \ref{decompositionVoevodsky} for precise statements.

\begin{theo}
If $r=2,3,$ $d=1,$ then the class of $M(r,L)$ has a decomposition of the same type as above in $\widehat{K_0}(\bVar)$ and $K_0(\widehat{\DM}_{\gm}).$
\end{theo}

In this article, we provide a uniform way to obtain the decompositions. In particular we found a very simple proof of Theorem \ref{rank2}. We expect this strategy will work for other cases. Let us briefly explain our method to provide new decomposition of the motivic Poincar\'e polynomials of moduli spaces.  Basically, we will use the classical strategy of Harder and Narasimhan \cite{Harder, HN} which now becomes a standard way to study cohomology groups of the moduli spaces of stable bundles. First, the class of the moduli stack $\Bun_{r,d}$ is computed in (various) completions of the Grothendieck rings of varieties, Chow motives or Voevodsky motives (cf. \cite{Bano;thesis, Bano2, BD, HL1, HL2}). We can compare the computation of the classes of $\Bun_{r,d}$ in $\widehat{K_0}(\bVar)$ (cf. \cite{BD}), $\widehat{K_0}(\Chow^{\eff})$ (cf. \cite{Bano;thesis, Bano2}) and $K_0(\widehat{\DM}_{\gm})$ (cf. \cite{HL1, HL2}) using the following description of motivic zeta function. (The following formula in the Grothendieck ring of varieties is known to experts (cf. \cite{GPHS, Heinloth}). See Proposition \ref{motiviczeta} for more details.)

\begin{prop}[Motivic zeta function of curve]
We have the following identity in $\widehat{K_0}(\Chow^{\eff}).$
$$ Z(C,t) = \frac{(1+t)^{h^1(C)}}{(1-t)(1-\LL t)} $$
\end{prop}

Then we will extract motivic terms corresponding to unstable vector bundles. We can compute these terms (cf. \cite{Bano;thesis, DR, EK, GPHS, HN}) in $\widehat{K_0}(\bVar),$ (or $\widehat{K_0}(\Chow^{\eff}),$ $K_0(\widehat{\DM}_{\gm})$) using Harder-Narasimhan filtrations. Then one can check that the inversion formula in \cite{Bano2} gives us the class of $M(r,L)$ in $\widehat{K_0}(\Chow^{\eff})$ (or $\widehat{K_0}(\bVar),$ $K_0(\widehat{\DM}_{\gm})$). Then we can compare our suggested motivic decomposition and the class of $M(r,L)$ computed using the above method and we can check that they are the same in $\widehat{K_0}(\Chow^{\eff})$ (or $\widehat{K_0}(\bVar),$ $K_0(\widehat{\DM}_{\gm})$). During our computation and comparison, we treat the terms corresponding to Jacobians and $C_k$ for $1 \leq k \leq g-1$ as if they are irreducible terms. This idea comes from the facts that derived categories of Jacobians and certain symmetric products of curves do not admit nontrivial semiorthogonal decomposition. See \cite{BGL} for precise statements and more details. To be more precise, we express motivic zeta function and motives of unstable bundles in terms of symmetric curves and Jacobians using the following motivic decompositions.

\begin{prop}\label{prop:symmpro-int}
(1) Let $g \leq k \leq 2g-2.$
Then we have the following isomorphism in $\Chow^{\eff}.$ 
\begin{equation}
   \label{eq:symmpro-int}
   h(C_k) \cong
   h(J(C)) \otimes (\bigoplus_{l=0}^{k-g} \LL^{ l})
   \oplus h(C_{2g-2-k}) \otimes \LL^{ k-g+1}  
\end{equation}
(2) Let $k \geq 2g-1.$ Then we have the following isomorphism in $\Chow^{\eff}.$
$$ h(C_k) = h(J(C)) \otimes (\bigoplus_{l=0}^{k-g} \LL^{ l}) $$
\end{prop}

The above motivic decomposition is inspired by a work of Toda (cf. \cite{Toda}) and proved using results of del Ba\~no (cf. \cite{Bano;thesis}). We can also obtain the same identity in $K_0(\bVar).$ Using the above decomposition, we can rewrite the motivic zeta function as follows.

\begin{prop}\label{DecZeta-int}
We have the following equality in $\widehat{K_0}(\Chow^{\eff}).$
$$ Z(C,\LL^{ i}) = \sum_{k=0}^{g-1} \chi(C_{k})  \mathbb{L}^{ ik} + \sum_{k'=0}^{g-2} \chi(C_{k'})  \LL^{ (2i+1)(g-1)-(i+1)k'} $$
$$ + \chi(J(C)) (\frac{\LL^{ ig}}{(1-\LL^{ i})(1-\LL^{ i+1})}) $$
\end{prop}

By extracting motives of unstable bundles we obtain the desired result.

\bigskip

It is an interesting question whether the motive of $M(r,L)$ can be expressed as a direct sum of motives of other varieties. Being motivated by the above strategy, we have the following conjectures.

\begin{conj}
The motive of $M(r,L)$ can be express as a direct sum of motives of products of symmetric products and Jacobian of $C.$
\end{conj}

\bigskip

Based on the above motivic decomposition and Orlov's conjecture, we have the following conjectures.

\begin{conj}\label{SODrank3}
(1) Let $M(3,L)$ be the moduli space of rank $3$ stable vector bundles on $C$ whose determinants are isomorphic to a fixed line bundle $L$ of degree $1$ on $C.$ Then the derived category of $M(3,L)$ has the following semiorthogonal decomposition
$$ D(M(3,L)) = \langle \cdots, D(C_{k_1} \times C_{k_2}), D(C_{k_1} \times C_{k_2}), \cdots, D(C_{g-1} \times C_{g-1}) \rangle $$
where 
$ (k_1,k_2)$ is a pair of nonnegative integers satisfying $k_1+k_2 < 2(g-1)$ or $k_1+k_2=2(g-1), k_1<g-1.$ \\
(2) The class of $M(r,L)$ has a decomposition of the same type as above in the category of Chow motives and Voevodsky motives. \\
(3) The Karoubian completion of Fukaya category of $M(3,L)$ has a decomposition of the same type as above.
\end{conj}

Especially, we expect there will be a corresponding decompositions of quantum cohomology groups of moduli spaces. For example, there is a formula appears in Floer cohomology which is compatible with the motivic decomposition in rank 2 case (cf. \cite[Conjecture 24]{Munoz}). See Section 7.3 and \cite{Lee} for more precise statements and details.

\bigskip

The content of this paper is as follows. In section 2, we recall several definitions and known results about motives and moduli spaces of vector bundles on curves which we will use in this paper. In section 3, we discuss motivic decompositions of symmetric products of curves and motivic zeta functions. In section 4, we obtain decomposition of the motivic Poincar\'e polynomial of $M(2,L)$ hence provide new and simpler proof of a version of the main result of \cite{Lee}. In section 5, we discuss motivic Poincar\'e polynomial of $M(3,L)$ and provide a new decomposition of it. In section 6, we discuss the same decompositions in the dimensional completion of Grothendieck ring of varieties and Voevodsky motives and several conjectures and further directions.

\bigskip

\bigskip

\noindent\textbf{Notation}. Let us fix an algebraically closed field $\KK$ whose characteristic is 0. Algebraic varieties, schemes and stacks we will consider in this paper are defined over $\KK.$ We will use $C_k$ to denote the $k$-th symmetric power of $C$ where $C$ is a smooth projective curve. Let $\SmProjVar$ (resp. $\bVar,$ $\Sch,$ $\Sch_c$) denotes the category of smooth projective varieties (resp. varieties, schemes, schemes with proper morphisms) over $\KK.$ We will use $\QQ$ for the coefficients of Chow and Voevodsky motives in this paper. Sometimes, we use abuse of notation if we think what the notation means is clear from the context. For example, we use the same notation to denote classes of a variety in $K_0(\bVar)$ and $\widehat{K_0}(\bVar).$

\bigskip

\section{Preliminaries}

Let us recall some basic notions and facts which we will use in this paper.

\subsection{Grothendieck ring of varieties and stacks}

\begin{defi}
Grothendieck ring of varieties $K_0(\bVar)$ is a ring generated by the symbol $[X]$ where $X$ is the isomorphism class of a variety in $\bVar$ modulo relations of the following form
$$ [X]=[X \setminus Z]+[Z] $$
where $Z$ is a closed variety in $X.$ The multiplication law is given by the Cartesian product. We use the notation $\LL$ to denote the class of affine line $[\A^1] \in K_0(\bVar).$
\end{defi}

We want to discuss motives of stacks. In order to do it, we discuss dimensional filtration of Grothendieck ring of varieties.

\begin{defi}\cite[dimensional completion]{BD, GPHS}
\label{def:dimensional-completion}
Let $F^m(K_0(\bVar))$ denotes the abelian subgroup of $K_0(\bVar)[\frac{1}{\LL}]$ generated by $\frac{[X]}{\LL^n}$ where $X$ is a variety such that dimension of $X$ minus $n$ is less than or equal to $-m.$ Let $\widehat{K_0}(\bVar)$ be the completion of $K_0(\bVar)$ with respect to the filtration $F^m(K_0(\bVar)).$
\end{defi}

Using the dimensional filtration of Grothendieck ring of varieties, we can define motivic classes of some stacks as follows.

\begin{defi}\cite{BD, GPHS}
(1) Let $GL(m)$ be a linear group acting on a variety $Z$ and $\cZ=[Z/GL(m)]$ be the quotient stack. We define the class of $[\cZ]$ in $\widehat{K_0}(\bVar)$ to be
$$ [\cZ]=\frac{[Z]}{[GL(m)]}. $$
(2) Let $G$ be an affine group acting on a variety $Z$ and $\cZ=[Z/G]$ be the quotient stack. We define the class of $[\cZ]$ in $\widehat{K_0}(\bVar)$ to be
$$ [\cZ]=\frac{[Z \times _{G} GL(m)]}{[GL(m)]} $$
where $G \to GL(m)$ is a faithful representation. \\
(3) A stack $\cX$ with linear (or affine) stabilizers is essentially of finite type if it admits a countable stratification $\cX = \bigsqcup_{i=0}^{\infty} \cZ_i$ where $\cZ_i$ is of finite type, locally closed with $\dim ~ \cZ_i$ goes to $-\infty$ as $i$ goes to $\infty.$ \\
(4) Let $\cX = \bigsqcup_{i=0}^{\infty} \cZ_i$ be a stack of essentially of finite type and suppose that each $\cZ_i$ is the global quotient of a variety $Z_i$ by $G_i.$ Then we define the class $[\cX]$ in $\widehat{K_0}(\bVar)$ to be
$$ [\cX] = \sum_{i=0}^{\infty} [\cZ_i] $$
where $ [\cZ_i]=\frac{[Z_i]}{[G_i]}. $
\end{defi}

\begin{rema}
From the definition of $\widehat{K_0}(\bVar),$ one can check that the class $\frac{1}{[GL(m)]}$ is well-defined.
\end{rema}

\subsection{Chow and Voevodsky motives}

Grothendieck suggested the notion of Chow motives and Manin studied its basic properties and applications in \cite{Manin}. From now on we will use $\Chow$ (respectively $\Chow^{\eff}$) to denote the category of (respectively effective) Chow motives. There is a natural contravariant functor
$$ h : \SmProjVar \to \Chow^{\eff} $$
and see \cite{Bano;thesis, Bano1, Manin} for more details about Chow motives. \\

The theory of Chow motives works well in the category of smooth projective varieties. There have been many attempts to generalize the notion of Chow motives to wider classes of schemes and stacks. Among these attempts, Hanamura, Levine and Voevodsky extended the notion of Chow motives independently. It turns out that their constructions are equivalent under certain mild assumptions. In this paper, we will use Voevodsky's construction. We will use $\DM_{\gm}^{\eff} := \DM_{\gm}^{\eff}(\KK,\QQ)$ to denote the effective geometric motives defined by Voevodsky (cf. \cite{Voevodsky}). Voevodsky proved that there is a fully faithful functor from the category of effective Chow motives to the category of effective geometric motives.

\begin{theo}\cite{Voevodsky}
There is a fully faithful functor from $\Chow^{\eff}$ to $\DM_{\gm}^{\eff}.$
\end{theo}

Indeed, he proved that there is a functor $ \Sch \to \DM_{\gm}^{\eff} $ which extends the previous functor $h$ defined on the category of smooth projective varieties. We will use the same $h$ to denote this functor. Moreover there is another functor
$$ h_c : \Sch_c \to \DM_{\gm}^{\eff} $$
which agrees with $h$ when we restrict it to the proper schemes.

It turns out that the above functors enjoy many nice properties. For example, we have a version of Poincar\'e duality. Let $M$ be an element in $\DM_{\gm}^{\eff}.$ Then the dual of $M$ is defined as follows.
$$ M^\vee:=Hom(M,\QQ(0)) $$

Let $X$ be a smooth scheme of dimension $d.$ Then we have the following isomorphism.
$$ h_c(X) = h(X)^\vee \otimes \LL^{ d} $$

Because there is a functor from $\Chow^{\eff}$ to $\DM_{\gm}^{\eff}$ it is natural to compare their Grothendieck groups. Bondarko proved that they are the same in \cite{Bondarko}.

\begin{theo}\cite{Bondarko} The Grothendieck group of effective Chow motives is isomorphic to the Grothendieck group of effective geometric motives, i.e.
$$ K_0(\Chow^{\eff}) = K_0(\DM_{\gm}^{\eff}). $$
\end{theo}

Moreover, there is a morphism from $K_0(\bVar)$ to $K_0(\Chow^{\eff})$ (hence also to $K_0(\DM_{\gm}^{\eff})$) when the characteristic of the base field is zero.

\begin{theo}\cite{MNP}
Then there is a unique morphism 
$$ \chi_c : K_0(\bVar) \to K_0(\Chow^{\eff}) $$
such that $\chi_c([X])=h(X)$ where $X$ is a smooth projective variety.
\end{theo}

\subsection{Motives of moduli stacks of vector bundles on curves}

Kapranov introduced the notion of motivic zeta function in \cite{Kapranov}.

\begin{defi} Let $C$ be a smooth projective curve. Let $C_j$ be the $j$-th symmetric product. Then we define the motivic zeta function as follows. 
$$ Z(C,t):= \sum_{j=0}^{\infty} h(C_j)  t^{ j} $$
\end{defi}

Especially, we have the following identity.

$$ Z(C,\LL^{ i})= \sum_{j=0}^{\infty} h(C_j)  \LL^{ ij} $$ 

Kapranov showed that the motivic zeta function of a curve is a rational function (cf. \cite{Kapranov}).

\begin{theo}\cite[Theorem 1.1.9]{Kapranov} Let $C$ be a curve of genus $g.$ Then $Z(C,t)$ has the following properties. \\
(1) [Rationality] $Z(C,t)$ is a rational function. To be more precise, $(1-t) (1-\LL t) Z(C,t) $ is a polynomial of degree $2g.$ \\
(2) [Functional equation] $$ Z(C,t) = \LL^{ g-1}t^{ 2g-2}Z(C,\LL^{-1}t^{-1}) $$ 
\end{theo}

We will discuss more about the motivic zeta functions of curves in Section 3.

\subsection{Moduli spaces of vector bundles on curves}

In \cite{HN}, Harder and Narasimhan introduced the notion of the Harder-Narasimhan filtration in order to count points of moduli space of stable vector bundles on $C$ defined over finite fields. Let us recall the notion.

\begin{defi}
Let $E$ be a vector bundle on $C.$ Then there is a unique filtration $0=E_0 \subset \cdots \subset E_n=E$ of $E$ satisfying the following properties. \\
(1) For each $i,$ $E_i/E_{i+1}$ is semistable bundle on $C.$ \\
(2) $\mu(E_1/E_0) > \cdots > \mu(E_i/E_{i+1}) \cdots > \mu(E_n/E_{n-1}). $ \\
Here $\mu$ denotes the slope stability function. We call the above unique filtration the Harder-Narasimhan filtration of $E.$
\end{defi}

In order to get an algebro-geometric way to compute cohomology groups of moduli stacks, Bifet, Ghione and Letizia studied matrix divisors in \cite{BGhL}. Let us recall their construction.

\begin{defi}
A matrix divisor of rank $r$ over $C$ is a rank $r$ vector bundle $E$ with an injective $\mathcal{O}_C$-modules
$$ \iota : E \to \mathcal{K}_C^{\oplus r} $$
where $\mathcal{K}_C$ is the constant sheaf of rational functions on $C.$
\end{defi}

Let $\iota : E \to \mathcal{K}_C^{\oplus r}$ be a matrix divisor. Then we can find an effective divisor $D$ such that $\iota$ factors via $\cO(D)^{\oplus r} \subset \mathcal{K}_C^{\oplus r}.$ For given $d$ and effective divisor $D,$ the matrix divisors of degree $d$ lying in $\mathcal{O}(D)$ are the $\mathbb{K}$-points of the Quot scheme $Q_{r,d}(D)$ parametrizing degree $r \mathrm{deg}(D)-d$ torsion sheaves which are quotients of $\mathcal{O}(D)^{\oplus r}.$

For two effective divisors $D \subset D'$ on $C$ we have a closed immersion $Q_{r,d}(D) \subset Q_{r,d}(D').$ Therefore $Q_{r,d}(D)$ defines a ind-variety and we call it the ind-variety of matrix divisors of rank $r$ and degree $d.$ From the Harder-Narasimhan filtration we have the Shatz stratification of $Q_{r,d}$ as follows
$$ Q_{r,d} = \bigcup_{(r,d)} Q_{(r,d)}^{ss} $$
where $(r,d)=(\cdots, r_i, d_i ,\cdots)$ is the Shatz polygon. See \cite{Bano;thesis, BGhL} for more details.

\subsection{Motivic Poincar\'e polynomial}

We will define motivic Poincar\'e polynomial of a motive $M$ following \cite{Bano;thesis, Bano1}. Let $K_0(\Chow^{\eff})$ be the Grothendieck ring of effective Chow motives over $\KK.$ Then we use $\widehat{K_0}(\Chow^{\eff})$ to denote the completion of $K_0(\Chow^{\eff})$ along $(\LL)$ the ideal generated by the Lefschetz motive. 

\begin{defi}\cite[Definition 6.14]{Bano;thesis}
We say that the motivic Poincar\'e polynomial of an ind-variety $(X_{i})_{i \in I}$ stabilizes if for each $m \in \NN$ there exists $i_m$ such that, for all $i>i_m$, $\chi(X_i)-\chi(X_{i_m})$ belongs to $\LL^{ m}.$ In this case we will use $\chi(X)$ to denote the element of $\widehat{K_0}(\Chow^{\eff})$ to which $\chi(X_{i})_{i \in I}$ is converging and call it the motivic Poincar\'e polynomial of $X.$
\end{defi}

Using matrix divisors and work in \cite{BGhL}, del Ba\~no obtained the following results. 

\begin{theo}\cite{Bano;thesis} \label{Bano;ind-variety}
(1) The motivic Poincar\'e polynomial of the ind-variety $Q_{r,d}$ converges to
$$ \frac{(1+1)^{h^1(C)}}{(1-\LL^{ n})} \prod_{i=1}^{r-1} \frac{(1+\LL^{ i})^{h^1(C)}}{(1-\LL^{ i})^{ 2}}. $$
(2) For each Shatz polygon $(r,d),$ the motivic Poincar\'e polynomial of $Q_{(r,d)}$ converges and $$ \chi (Q_{(r,d)}^{ss}) = \prod_i \chi (Q_{r_i,d_i}^{ss}). $$
(3) If $r$ and $d$ are coprime we have
$$ \chi (Q_{r,d}^{ss}) = \frac{1}{1-\LL} \chi (M(r,d)). $$
\end{theo}

Using these results and Laumon and Rapoport's inversion formula, del Ba\~no obtained the following formula.

\begin{theo}\cite{Bano;thesis} \cite[Theorem 4.11]{Bano1} \label{Bano;inversion}
Let $r \geq 2, d$ be two integers which are coprime to each other. Then the motivic Poincar\'e polynomial of $M(r,d)$ is
$$ \sum_{s=1}^n \sum_{n_1+\cdots+n_s=n, n_i \in \NN} (-1)^{s-1} \frac{((1+1)^{h^1(C)})^s}{(1-\LL)^{s-1}} \prod_{j=1}^s \prod_{i=1}^{n_j-1} \frac{(1+\LL^{ i})^{h^1(C)}}{(1-\LL^{ i})(1-\LL^{ i+1})} $$
$$  \prod_{j=1}^{s-1}\frac{1}{1-\LL^{ n_j+n_{j+1}}}  \LL^{ (\sum_{i < j}n_in_j(g-1)+\sum_{i=1}^{s-1}(n_i+n_{i+1})\langle -\frac{n_1+\cdots+n_i}{n}d \rangle)}. $$
\end{theo}

\subsection{Completions of Grothendieck rings of Voevodsky motives}

Because of Eilenberg's swindle we see that $K_0(\DM)=0.$ Therefore we need a version of completion of Voevodsky motives in order to obtain meaningful class of $M(r,L)$ using the strategy of \cite{Harder, HN}. Indeed, there are several known versions of completions of Grothendieck ring of motives. We will briefly recall a construction sketched in \cite{HL2}. Let $\DM$ be the symmetric monoidal stable $\infty$-category of Voevodsky motives.

\begin{defi}\cite{HL1}
The dimensional filtration of $\DM$ is a $\ZZ$-indexed filtration 
$$ \cdots \subset \DM_{d} \subset \DM_{d+1} \subset \cdots $$
where $\DM_d$ denotes the smallest localizing subcategory of $\DM$ containing $h_c(X)(d')$ for all separated schemes $X$ of finite type over $\KK$ and all integers $d'$ with $dim(X)+d' \leq d.$
\end{defi}

We can define the dimensional completion of $\DM_{\gm}$ as follows.

\begin{defi}
\label{def:completion-dm}
$$ \widehat{\DM}_{\gm}:=\lim_{d \in \NN} \DM_{\gm} / F^d(\DM_{\gm}) $$
\end{defi}

There is a ring homomorphism $\chi_c : \widehat{K_0}(\bVar) \to K_0(\widehat{\DM}_{\gm})$ induced by the functor $\bVar \to \DM_{\gm}.$ Therefore we can compare the classes of $\Bun_{r,d}$ in both sides. First, Behrend and Dhillon computed the class of $\Bun_{r,d}$ as follows.

\begin{theo}[Behrend-Dhillon] In $\widehat{K_0}(\bVar),$ the class of $\Bun_{\mathrm{SL_r}}$ is as follows
\begin{eqnarray*}
 [\Bun_{\mathrm{SL_r}}]&=& \LL^{(r^2-1)(g-1)} \prod_{i=2}^{r} Z(C,\LL^{-i} )
\end{eqnarray*}
where $Z(C,t):=\sum_{k \geq 0}[C_k] t^k.$
\end{theo}

The second equality follows from the functional equation of the motivic zeta function $Z(C,t)$.
On the other hand, Hoskins and Lehalleur computed the class of $\Bun_{r,d}$ in the category of Voevodsky's motives.

\begin{theo}\cite[Theorem 1.1]{HL2}
We have the following isomorphism in $\DM.$
$$ h(\mathrm{Bun}_{r,d}) = h(Jac(C)) \otimes h(B\GG_m) \otimes \bigotimes_{i=1}^{r-1} Z(C,\LL^{\otimes i}) $$
\end{theo}

Combining \cite[Lemma 4.4]{HL1} and \cite[Theorem 1.1]{HL2}, we see that the two descriptions are compatible to each other.

\begin{theo}[Hoskins-Lehalleur]
We have the following identity in $K_0(\widehat{\DM}_{\gm}).$
$$ \chi_c([\Bun_{r,d}]) = [h_c(\Bun_{r,d})] $$
\end{theo}

In this paper, we will compute the classes of $M(2,L)$ and $M(3,L)$ in $\widehat{K
_0}(\bVar),$ $\widehat{K_0}(\Chow^{\eff})$ and $K_0(\widehat{\DM}_{\gm})$ using the above result.

\begin{rema}
Note that we are using different completions for $\widehat{K_0}(\Chow^{\eff})$ and $\widehat{K_0}(\bVar)$ (or $K_0(\widehat{\DM}_{\gm})$). In $\widehat{K_0}(\Chow^{\eff})$ we are using the completion with respect to the ideal generated by $\LL$, whereas in $\widehat{K_0}(\bVar)$ we are using dimensional completion (Definition \ref{def:dimensional-completion}), and in $K_0(\widehat{\DM}_{\gm})$ we use Definition \ref{def:completion-dm}.
\end{rema}

\bigskip

\section{symmetric products of curves}

In this section we study motives of symmetric products of curves. See \cite{Bano1, Bano2} for notations and backgrounds about them.

\subsection{$\lambda$-structure}

Let $\cC$ be a $\QQ$-linear pseudo-abelian tensor category.

\begin{defi}\cite{Bano1}
A $\lambda$-structure on $\cC$ is a sequence of functors $\lambda^n : \cC \to \cC$ for $n \in \NN$ such that \\
(1) $\lambda^0$ is the constant functor sending every object to $1$ and every morphism to $Id_{1},$ \\
(2) $\lambda^1$ is the identity functor, \\
(3) there are natural isomorphisms $\lambda^n(X \oplus Y) = \oplus_{a+b=n}\lambda^a(X)  \lambda^b(Y).$ 
\end{defi}

Because $\cC$ is a tensor category, there is a morphism $ \phi(\sigma) : M^{ n} \to M^{ n}$ for every $M \in \cC$ and $\sigma \in S_n.$ One can check that $\frac{1}{n!} \sum_{\sigma \in S_n}\phi(\sigma) : M^{ n} \to M^{ n}$ is a projector. See \cite{Bano1} for more details.

\begin{defi}
For each $M \in \cC$ and $n \in \NN$ we define
$$ \lambda^n(M)=(M^{ n}, \frac{1}{n!} \sum_{\sigma \in S_n}\phi(\sigma) ) $$
\end{defi}

S. del Ba\~no showed that the above functors define a $\lambda$-structure on $\cC.$

\begin{theo}\cite[Theorem 3.4]{Bano1}
The functor $\lambda^n$ defined above gives a $\lambda$-structure on $\cC.$
\end{theo}

\subsection{Motives of symmetric products of curves}

In \cite{Bano1}, S. del Ba\~no proved the following motivic decomposition.

\begin{prop}\cite{Bano1}\label{prop:bano-ck}
Let $C_k$ be the $k$-th symmetric power of $C.$ Then there is the following decomposition.
$$ h(C_k) = \sum_{a+b+c=k} 1^{\otimes a} \otimes \lambda^b h^1(C) \otimes \LL^{\otimes c} = \sum_{b+c\leq k} \lambda^b h^1(C) \otimes \LL^{\otimes c}$$
where $a,b,c$ are nonnegative integers.
\end{prop}

In order to discuss motives of symmetric products of a curve and motive of its Jacobian, we recall the following definition.

\begin{defi} Let $C$ be a smooth projective curve of genus $g.$ Then we define $(1+t)^{h^1(C)}$ as follows.
$$ (1+t)^{h^1(C)} := \sum_{a=0}^{2g} \lambda^a h^1(C)  t^{ a} $$
\end{defi}

In particular, we define $(1+\LL^{\otimes m})^{h^1(C)}$ as follows.
\begin{rema}
$$ (1+\LL^{\otimes m})^{h^1(C)} := \sum_{a=0}^{2g} \lambda^a h^1(C)  \LL^{\otimes ma} $$
\end{rema}

Using the above definition we can express the motive of Jacobian.

\begin{rema}
  \label{rmk:bano-jacobian}
  \cite{Bano;thesis}
The motive of Jacobian of $C$ is as follows.
$$ h(J(C)) = (1+1)^{h^1(C)} = \sum_{a=0}^{2g} \lambda^a h^1(C)$$
\end{rema}

From the Poincar\'e duality we have the following isomorphisms.

\begin{prop}\label{Duality of Jacobian}
  We have the following isomorphism for $\delta$ integer
  with $0 \leq \delta\leq g$
  $$
  \lambda^{g+\delta} h^1(C)
  \cong
  \lambda^{g-\delta} h^1(C) \otimes \LL^{ \delta}
  $$
 and hence
 \begin{equation}
 \lambda^{g+\delta} h^1(C) \otimes \LL^{ c}
 \cong
 \lambda^{g-\delta} h^1(C) \otimes \LL^{ \delta+c}
 \end{equation}
\end{prop}
\begin{proof}
From Poincar\'e duality we have the following isomorphism.
$$ h_c(C) \cong h(C)^\vee \otimes \LL $$
Because $C$ is proper we have the following isomorphism.
$$ h^1(C) \cong h^1(C)^\vee \otimes \LL $$
Again, we have the following isomorphism due to Poincar\'e duality.
$$ h_c(J(C)) \cong h(J(C))^\vee \otimes \LL^{ g} $$
Because $J(C)$ is proper we have the following isomorphism.
$$ h(J(C)) \cong \sum_{a=0}^{2g} \lambda^a h^1(C) \cong h(J(C))^\vee \otimes \LL^{ g} \cong (\sum_{a=0}^{2g} \lambda^a h^1(C))^\vee \otimes \LL^{ g} $$
It is well-known that the motive of Abelian variety can be decompose into summands, we can compare each summand of this isomorphism. (For the uniqueness of the canonical decomposition of the motive of an abelian variety, see \cite{Scholl} and references therein.) Therefore we obtain the desired result. The last formula is obtained by tensoring with $\LL^{\otimes c}$.
\end{proof}

\subsection{Motivic decompositions of motives of symmetric products of curves and zeta functions}

Now we want to compare the motives of symmetric products and the Jacobian of a curve. In the level of derived categories of coherent sheaves, Toda obtain the following semiorthogonal decompositions in \cite{Toda}.

\begin{theo}\cite{Toda}
Let $g \leq k \leq 2g-2.$ Then $D(C_k)$ has following semiorthogonal decomposition
$$ D(C_k) = \langle D(J(C)), \cdots, D(J(C)), D(C_{2g-2-k}) \rangle $$
where there are $k-g+1$ copies of $D(J(C)).$
\end{theo}

Being motivated by the above decomposition and Orlov's conjecture we obtain the following isomorphism.

\begin{prop}\label{prop:symmpro}
(1) Let $g \leq k \leq 2g-2.$ Then we have the following isomorphism
in $\Chow^{\eff}$.
\begin{equation}
   \label{eq:symmpro}
   h(C_k) \cong
   h(J(C)) \otimes (\bigoplus_{l=0}^{k-g} \LL^{\otimes l})
   + h(C_{2g-2-k}) \otimes \LL^{\otimes k-g+1}  
\end{equation}
(2) Let $k \geq 2g-1.$ Then we have the following isomorphism in $\Chow^{\eff}$.
$$ h(C_k) \cong h(J(C)) \otimes (\bigoplus_{l=0}^{k-g} \LL^{\otimes l}) $$
\end{prop}
\begin{proof}
The second part is a direct consequence of the fact that $C_k$ is a projective bundle over $J(C)$ when $k \geq 2g-1.$ \\

Now let us prove the first part. We will first show that both sides of \eqref{eq:symmpro} decompose into direct sums of the form $\lambda^b h^1(C) \otimes \LL^{\otimes c}$, so the proof will be a combinatorial argument to identify the summands on both sides. \\

Let $g \leq k \leq 2g-2.$ From del Ba\~no's description (Proposition \ref{prop:bano-ck}) we have the following isomorphism. 
\begin{equation}
  \label{eq:ck}
h(C_{k}) = \bigoplus_{b+c \leq k}  \lambda^b h^1(C) \otimes \LL^{ c}
\end{equation}
Combining the formula for $h(J(C))$ (Remark \ref{rmk:bano-jacobian}) with Poincar\'e duality (Proposition \ref{Duality of Jacobian}), we see that
$$
h(J(C)) = \bigoplus_{a=0}^{g} \lambda^a h^1(C) \oplus
\bigoplus_{a=g+1}^{2g} \lambda^{2g-a} h^1(C) \otimes \LL^{\otimes a-g}
$$

If we plot a dot $(b,c)$ in the plane $\NN^2$ for each summand of the
form $\lambda^bh^1(C) \otimes \LL^{\otimes c}$, we obtain:

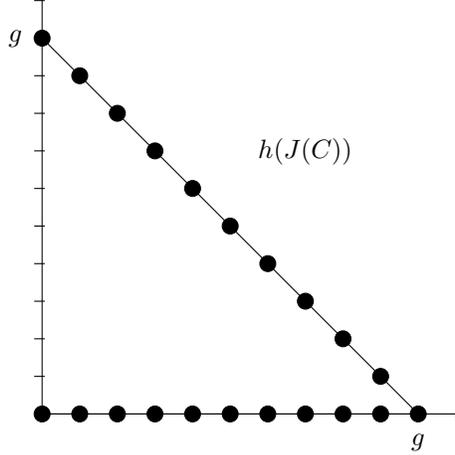
\begin{tikzpicture}[y=.5cm, x=.5cm,font=\sffamily]
\draw (0,0) -- coordinate (x axis mid) (11,0);
\draw (0,0) -- coordinate (y axis mid) (0,11);
\foreach \x in {0,...,11}
	\draw (\x,1pt) -- (\x,-3pt);
\draw (10,-4pt) node[anchor=north] {$g$};
\foreach \y in {0,...,11}
\draw (1pt,\y) -- (-3pt,\y) ;
\draw (-4pt, 10) node[anchor=east] {$g$};

\draw plot[mark=*, mark size=3pt] coordinates {
(0,0) (1,0) (2,0) (3,0) (4,0) (5,0) (6,0) (7,0) (8,0) (9,0) (10,0)
(9,1) (8,2) (7,3) (6,4) (5,5) (4,6) (3,7) (2,8) (1,9) (0,10)};
\node[] at (7,7) {$h(J(C))$};

\end{tikzpicture}

In the same way, the sum
$$
h(J(C)) \otimes  (\bigoplus_{l=0}^{k-g} \LL^{\otimes l})
$$
will produce the diagram:

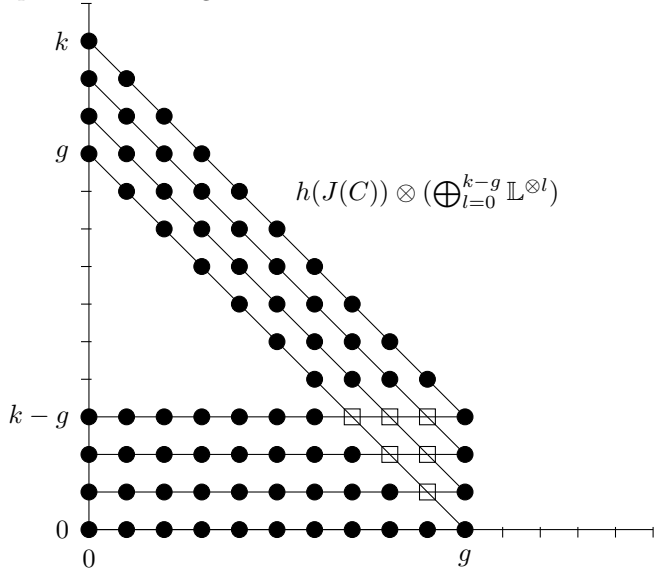
\begin{tikzpicture}[y=.5cm, x=.5cm,font=\sffamily]
\draw (0,0) -- coordinate (x axis mid) (15,0);
\draw (0,0) -- coordinate (y axis mid) (0,14);
\foreach \x in {0,...,15}
\draw (\x,1pt) -- (\x,-3pt);
\draw (10,-4pt) node[anchor=north] {$g$};
\draw (0,-4pt) node[anchor=north] {$0$};

\foreach \y in {0,...,14}
\draw (1pt,\y) -- (-3pt,\y) ;
\draw (-4pt, 10) node[anchor=east] {$g$};
\draw (-4pt, 13) node[anchor=east] {$k$};
\draw (-4pt, 3) node[anchor=east] {$k-g$};
\draw (-4pt, 0) node[anchor=east] {$0$};

\node[] at (9,9) {$h(J(C)) \otimes (\bigoplus_{l=0}^{k-g} \LL^{\otimes l})$};

        \draw plot[mark=*, mark size=3pt] coordinates {
        (0,0)
          (1,0)
          (2,0)
          (3,0)
          (4,0)
          (5,0)
          (6,0)
          (7,0)
          (8,0)
          (9,0)
          (10,0)
          (6,4)
          (5,5)
          (4,6)
          (3,7)
          (2,8)
          (1,9)
          (0,10)};

        \draw plot[mark=*, mark size=3pt] coordinates {
        (0,1)
        (1,1)
        (2,1)
        (3,1)
        (4,1)
        (5,1)
        (6,1)
        (7,1)
        (8,1)
        (10,1)
        (7,4)
        (6,5)
        (5,6)
        (4,7)
        (3,8)
        (2,9)
        (1,10)
        (0,11)};

        \draw plot[mark=*, mark size=3pt] coordinates {
        (0,2)
        (1,2)
        (2,2)
        (3,2)
        (4,2)
        (5,2)
        (6,2)
        (7,2)
        (10,2)
        (8,4)
        (7,5)
        (6,6)
        (5,7)
        (4,8)
        (3,9)
        (2,10)
        (1,11)
        (0,12)};

        \draw plot[mark=*, mark size=3pt] coordinates {
        (0,3)
        (1,3)
        (2,3)
        (3,3)
        (4,3)
        (5,3)
        (6,3)
        (10,3)
        (9,4)
        (8,5)
        (7,6)
        (6,7)
        (5,8)
        (4,9)
        (3,10)
        (2,11)
        (1,12)
        (0,13)};
        \draw plot[mark=square,mark size=3pt, only marks] coordinates {
          (7,3) (8,2) (8,3) (9,1) (9,2) (9,3)};
\end{tikzpicture}

Note that some of the points (drawn with empty squares) are duplicated. By formula \eqref{eq:ck}, the diagram corresponding to the motive $h(C_{2g-2-k})$ of the symmetric product of a curve is just the triangle defined by both axis and the line $b+c=2g-2-k$. Therefore, adding the last term, we obtain the diagram of the right hand side of \eqref{eq:symmpro}:

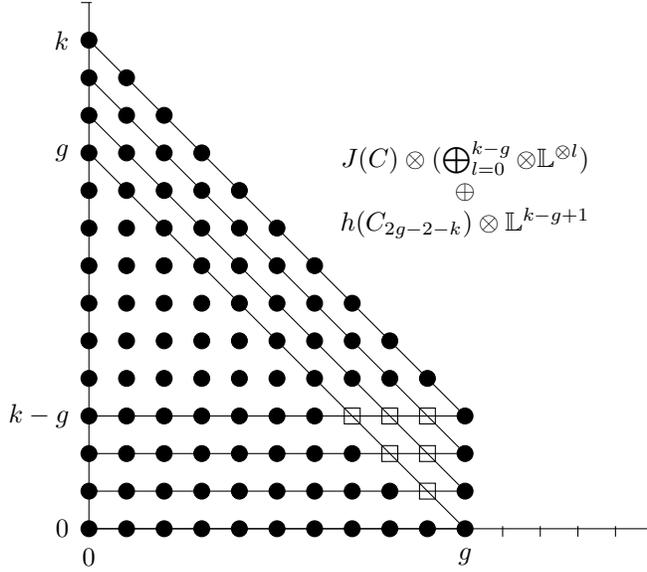
\begin{tikzpicture}[y=.5cm, x=.5cm,font=\sffamily]
\draw (0,0) -- coordinate (x axis mid) (15,0);
\draw (0,0) -- coordinate (y axis mid) (0,14);
\foreach \x in {0,...,15}
\draw (\x,1pt) -- (\x,-3pt);
\draw (10,-4pt) node[anchor=north] {$g$};
\draw (0,-4pt) node[anchor=north] {$0$};

\foreach \y in {0,...,14}
\draw (1pt,\y) -- (-3pt,\y) ;
\draw (-4pt, 10) node[anchor=east] {$g$};
\draw (-4pt, 13) node[anchor=east] {$k$};
\draw (-4pt, 3) node[anchor=east] {$k-g$};
\draw (-4pt, 0) node[anchor=east] {$0$};

\node[] at (10,9) {
  $\begin{array}{c}
    J(C) \otimes (\bigoplus_{l=0}^{k-g} \otimes \LL^{\otimes l}) \\
    \oplus \\
    h(C_{2g-2-k}) \otimes \LL^{ k-g+1}
  \end{array}$
};

        \draw plot[mark=*, mark size=3pt] coordinates {
        (0,0)
          (1,0)
          (2,0)
          (3,0)
          (4,0)
          (5,0)
          (6,0)
          (7,0)
          (8,0)
          (9,0)
          (10,0)
          (6,4)
          (5,5)
          (4,6)
          (3,7)
          (2,8)
          (1,9)
          (0,10)};

        \draw plot[mark=*, mark size=3pt] coordinates {
        (0,1)
        (1,1)
        (2,1)
        (3,1)
        (4,1)
        (5,1)
        (6,1)
        (7,1)
        (8,1)
        (10,1)
        (7,4)
        (6,5)
        (5,6)
        (4,7)
        (3,8)
        (2,9)
        (1,10)
        (0,11)};

        \draw plot[mark=*, mark size=3pt] coordinates {
        (0,2)
        (1,2)
        (2,2)
        (3,2)
        (4,2)
        (5,2)
        (6,2)
        (7,2)
        (10,2)
        (8,4)
        (7,5)
        (6,6)
        (5,7)
        (4,8)
        (3,9)
        (2,10)
        (1,11)
        (0,12)};

        \draw plot[mark=*, mark size=3pt] coordinates {
        (0,3)
        (1,3)
        (2,3)
        (3,3)
        (4,3)
        (5,3)
        (6,3)
        (10,3)
        (9,4)
        (8,5)
        (7,6)
        (6,7)
        (5,8)
        (4,9)
        (3,10)
        (2,11)
        (1,12)
        (0,13)};

        \draw plot[mark=square,mark size=3pt, only marks] coordinates {
(7,3) (8,2) (8,3) (9,1) (9,2) (9,3)};

        \draw plot[mark=*, mark size=3pt, only marks] coordinates {
          (0,4) (1,4) (2,4) (3,4) (4,4) (5,4)
          (0,5) (1,5) (2,5) (3,5) (4,5)
          (0,6) (1,6) (2,6) (3,6)
          (0,7) (1,7) (2,7) 
          (0,8) (1,8) 
          (0,9) 
        };
        
\end{tikzpicture}

Note that the duplicated points are exactly those in the isosceles triangle 
whose vertices have coordinates

$$
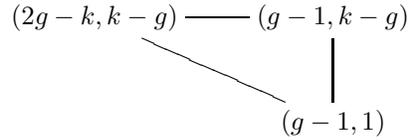
\xymatrix{
  {(2g-k,k-g)} \ar@{-}[r] \ar@{-}[rd]& {(g-1,k-g)} \ar@{-}[d]\\
  & {(g-1,1)}
}
$$

and using Poincar\'e duality, these points are mapped to the triangle

$$
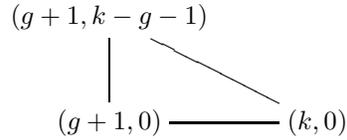
\xymatrix{
  {(g+1,k-g-1)} \ar@{-}[d] \ar@{-}[rd]\\
   {(g+1,0)} \ar@{-}[r]  & {(k,0)}
}
$$

Therefore, the previous diagram becomes:

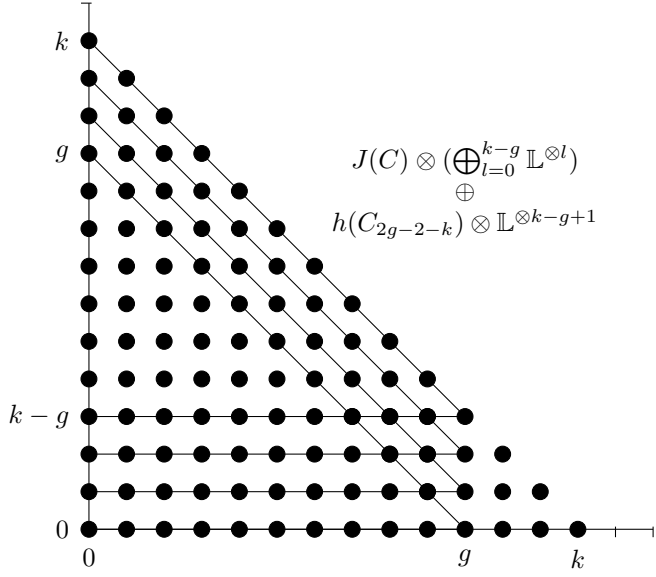
\begin{tikzpicture}[y=.5cm, x=.5cm,font=\sffamily]
\draw (0,0) -- coordinate (x axis mid) (15,0);
\draw (0,0) -- coordinate (y axis mid) (0,14);
\foreach \x in {0,...,15}
\draw (\x,1pt) -- (\x,-3pt);
\draw (0,-4pt) node[anchor=north] {$0$};
\draw (10,-4pt) node[anchor=north] {$g$};
\draw (13,-4pt) node[anchor=north] {$k$};

\foreach \y in {0,...,14}
\draw (1pt,\y) -- (-3pt,\y) ;
\draw (-4pt, 10) node[anchor=east] {$g$};
\draw (-4pt, 13) node[anchor=east] {$k$};
\draw (-4pt, 3) node[anchor=east] {$k-g$};
\draw (-4pt, 0) node[anchor=east] {$0$};

\node[] at (10,9) {
  $\begin{array}{c}
    J(C) \otimes (\bigoplus_{l=0}^{k-g} \LL^{\otimes l})\\
    \oplus \\
    h(C_{2g-2-k}) \otimes \LL^{\otimes k-g+1}
  \end{array}$
};

        \draw plot[mark=*, mark size=3pt] coordinates {
        (0,0)
          (1,0)
          (2,0)
          (3,0)
          (4,0)
          (5,0)
          (6,0)
          (7,0)
          (8,0)
          (9,0)
          (10,0)
          (9,1)
          (8,2)
          (7,3)
          (6,4)
          (5,5)
          (4,6)
          (3,7)
          (2,8)
          (1,9)
          (0,10)};

        \draw plot[mark=*, mark size=3pt] coordinates {
        (0,1)
        (1,1)
        (2,1)
        (3,1)
        (4,1)
        (5,1)
        (6,1)
        (7,1)
        (8,1)
        (9,1)
        (10,1)
        (9,2)
        (8,3)
        (7,4)
        (6,5)
        (5,6)
        (4,7)
        (3,8)
        (2,9)
        (1,10)
        (0,11)};

        \draw plot[mark=*, mark size=3pt] coordinates {
        (0,2)
        (1,2)
        (2,2)
        (3,2)
        (4,2)
        (5,2)
        (6,2)
        (7,2)
        (8,2)
        (9,2)
        (10,2)
        (9,3)
        (8,4)
        (7,5)
        (6,6)
        (5,7)
        (4,8)
        (3,9)
        (2,10)
        (1,11)
        (0,12)};

        \draw plot[mark=*, mark size=3pt] coordinates {
        (0,3)
        (1,3)
        (2,3)
        (3,3)
        (4,3)
        (5,3)
        (6,3)
        (7,3)
        (8,3)
        (9,3)
        (10,3)
        (9,4)
        (8,5)
        (7,6)
        (6,7)
        (5,8)
        (4,9)
        (3,10)
        (2,11)
        (1,12)
        (0,13)};


        \draw plot[mark=*, mark size=3pt, only marks] coordinates {
          (0,4) (1,4) (2,4) (3,4) (4,4) (5,4)
          (0,5) (1,5) (2,5) (3,5) (4,5)
          (0,6) (1,6) (2,6) (3,6)
          (0,7) (1,7) (2,7) 
          (0,8) (1,8) 
          (0,9) 
        };

         \draw plot[mark=*, mark size=3pt, only marks] coordinates {
           (11,0) (12,0) (13,0) 
           (11,1) (12,1) 
           (11,2) 
         };
\end{tikzpicture}

This is the diagram corresponding to $h(C_k)$ (see \eqref{eq:ck}), finishing the proof.

\end{proof}

There is a similar decomposition in $K_0(\bVar).$

\begin{prop}\label{prop:symmpro in Grothendieck ring}
We have the following equations in $K_0(\bVar)$. \\
(1) If $g \leq k \leq 2g-2$, then
\begin{equation}
   \label{eq:symmpro in Grothendieck ring}
   [C_k] = 
   [J(C)]  (\sum_{l=0}^{k-g} \LL^{ l})
   + [C_{2g-2-k}]  \LL^{ k-g+1}  
\end{equation}
(2) If $k \geq 2g-1$, then
$$ [C_k] = [J(C)]  (\sum_{l=0}^{k-g} \LL^{ l}) $$
\end{prop}
\begin{proof}
Again, the second part is a direct consequence of the fact that $C_k$ is a
projective bundle over $J(C)$ when $k \geq 2g-1.$ Therefore, we may
assume $g \leq k \leq 2g-2.$ \\

Let $X$ be a variety and $P_0=\PP(E_0)\to X$ and $P_1=\PP(E_1)\to X$ 
be projective
bundles where the vector bundles $E_0$ and $E_1$ have ranks $a_0+1$ and $a_1+1$ respectively, with $a_0>a_1$. The formula for the motive of a projective bundle gives
\begin{eqnarray}
[P_0] &=& [X]  \sum_{l=0}^{a_0}\LL^{ l}  = [X] \sum_{l=0}^{a_0-a_1-1}\LL^{ l} +
[X] \sum_{l=a_0-a_1}^{a_0}\LL^{ l} \nonumber\\
&=& [X] \sum_{l=0}^{a_0-a_1-1} \LL^{ l} +
\LL^{ a_0-a_1}  [X] \sum_{l=0}^{a_1}\LL^{ l} =  [X] \sum_{l=0}^{a_0-a_1-1}\LL^{ l} + \LL^{a_0-a_1}  [P_1] \label{motproj}
\end{eqnarray} 
We remark that this formula remains valid if $a_1=-1$, using the convention
that, in this case, $P_1$ is empty.

Consider the diagram
\begin{equation}
  \label{eq:dualsym}
\xymatrix{
  {C_k} \ar[rd]^{p} & & {C_{2g-2-k}} \ar[ld]_{q}\\
  & \Pic^{k}(C) = J(C)&
}
\end{equation}
where $p$ is the Abel-Jacobi map sending a cycle $Z$ to the line
bundle $\sO_C(Z)$ and $q$ sends the cycle $Z$ to
$K_C  \otimes \sO_C(-Z)$, where $K_C$ is the canonical line bundle of
$C$. Therefore, the fiber of $p$ over $L$ is
$\PP(H^0(L))$ and the fiber of $q$ is $\PP(H^1(L)^\vee)$.
These morphisms are projectivizations of vector bundles when restricted to
a Brill-Noether stratum $B_i=\{L:h^0(L)=i\}$ 
$$ \Pic^k(C) = B_1 \sqcup \cdots \sqcup B_l $$
Note that $h^0(L)\geq \deg(L) +1-g=k+1-g\geq 1$, 
since we are assuming $g\leq k$,
and also $h^0(L)=h^1(L)+k+1-g>h^1(L)$, so we can apply the formula
\eqref{motproj} to $p^{-1}(B_i)$ and $q^{-1}(B_i)$. 

Using the additivity of motives on stratifications and \eqref{motproj} we obtain the following 
$$ [C_k] = \sum_{i=1}^l [ p^{-1}(B_i) ] = 
\sum_{i=1}^l \Big( [ B_i ]  ( \sum_{l=0}^{k-g} \LL^{ l} ) +  [q^{-1}(B_i)]  \LL^{ k-g+1} \Big)$$
$$ = [ J(C) ]  ( \sum_{l=0}^{k-g} \LL^{ l} ) +  [ C_{2g-2-k} ]  \LL^{ k-g+1} $$
which is the desired result.
\end{proof}

Both propositions imply the following Corollary.

\begin{coro}\label{coro:symmpro}
The following formulas hold in $ K_0(\Chow^{\eff}) = K_0(\DM_{\gm}^{\eff}).$ \\
(1) 
If $g \leq k \leq 2g-2$, then
\begin{equation}
   \label{eq:symmpro in coro}
   \chi(C_k) =
   \chi(J(C))  (\sum_{l=0}^{k-g} \LL^{ l})
   + \chi(C_{2g-2-k})  \LL^{ k-g+1}  
\end{equation}
(2) If $k \geq 2g-1$, then 
$$ \chi(C_k) = \chi(J(C))  (\sum_{l=0}^{k-g} \LL^{ l}) $$
\end{coro}

From the motivic decompositions of symmetric products we obtain the following decompositions of motivic zeta function.

\begin{prop}\label{DecZeta-chow}
We have the following equality in $\widehat{K_0}(\Chow^{\eff}).$
$$ Z(C,\LL^{ i}) = \sum_{k=0}^{g-1} \chi(C_{k})  \mathbb{L}^{ ik} + \sum_{k'=0}^{g-2} \chi(C_{k'})  \LL^{ (2i+1)(g-1)-(i+1)k'} $$
$$ + \chi(J(C))  (\frac{\LL^{ ig}}{(1-\LL^{ i})(1-\LL^{ i+1})}) $$
\end{prop}
\begin{proof}
Using Corollary \ref{coro:symmpro} we have the following.
$$ Z(C,\LL^{ i}) = \sum_{k=0}^{g-1} \chi(C_{k})  \mathbb{L}^{ ik} + \sum_{k=g}^{2g-2} \chi(C_{k})  \mathbb{L}^{ ik} + \sum_{k=2g-1}^{\infty} \chi(C_{k})  \mathbb{L}^{ ik} $$
$$ = \sum_{k=0}^{g-1} \chi(C_{k})  \mathbb{L}^{ ik} + \sum_{k=g}^{2g-2} (\chi(J(C))  (\sum_{l=0}^{k-g} \LL^{ l}) + \chi(C_{2g-2-k})  \LL^{ k-g+1})  \mathbb{L}^{ ik} + \sum_{k=2g-1}^{\infty} \chi(J(C))  (\sum_{l=0}^{k-g} \LL^{ l})  \mathbb{L}^{ ik} $$
$$ = \sum_{k=0}^{g-1} \chi(C_{k})  \mathbb{L}^{ ik} + \sum_{k=g}^{2g-2} \chi(C_{2g-2-k})  \LL^{ (i+1)k-g+1} + \sum_{k=g}^{\infty} \chi(J(C))  (\sum_{l=0}^{k-g} \LL^{ l})  \mathbb{L}^{ ik} $$
$$ = \sum_{k=0}^{g-1} \chi(C_{k})  \mathbb{L}^{ ik} + \sum_{k'=0}^{g-2} \chi(C_{k'})  \LL^{ (2i+1)(g-1)-(i+1)k'} + \sum_{k=g}^{\infty} \chi(J(C))  (\frac{1-\LL^{ k-g+1}}{1-\LL})  \mathbb{L}^{ ik} $$
$$ = \sum_{k=0}^{g-1} \chi(C_{k})  \mathbb{L}^{ ik} + \sum_{k'=0}^{g-2} \chi(C_{k'})  \LL^{ (2i+1)(g-1)-(i+1)k'} + \chi(J(C))  (\frac{\LL^{ ig}}{(1-\LL^{ i})(1-\LL^{ i+1})}) $$
\end{proof}

By taking Euler characteristic we have the following Corollary.

\begin{coro}
We have the same decompositions for motivic Poincar\'e polynomials.
\end{coro}

Using a similar argument we have the following.

\begin{prop}\label{DecZeta-var}
We have the following equality in $\widehat{K_0}(\bVar).$
$$ \LL^{(2i-1)(g-1)} Z(C,\LL^{-i}) = \sum_{k=0}^{g-1} [C_{k}] \mathbb{L}^{(i-1)k} + \sum_{k'=0}^{g-2} [C_{k'}]  \LL^{ (2i-1)(g-1)-ik'} $$
$$ + [J(C)]  (\frac{\LL^{ (i-1)g}}{(\LL^{ i-1 }-1)(\LL^{ i }-1)}) $$
\end{prop}

Last but not least, we need the following description of motivic zeta function to compare computation of the classes of $\Bun_{r,d}$ in the completion of Grothendieck groups of Chow motives in \cite{Bano;thesis, Bano2} and Voevodsky motives in \cite{HL1, HL2}. The following formula in the Grothendieck ring of varieties is known to experts (cf. \cite{GPHS, Heinloth}). 

\begin{prop}[Motivic zeta function]\label{motiviczeta}
We have the following identity.
$$ Z(C,t) = \frac{(1+t)^{h^1(C)}}{(1-t)(1-\LL t)} $$
\end{prop}
\begin{proof}
Kapranov proved that $(1-t)(1-\LL t)Z(C,t)$ is a polynomial of degree $2g$ in \cite{Kapranov}. One can check that the constant coefficient is $1$ and the coefficient of $t$ is $h(C)-1-\LL=h^1(C)=\lambda^1 h^1(C).$ For $k \geq 2,$ the coefficient of the term $t^k$ is $h(C_k)-(1+\LL)h(C_{k-1})+\LL h(C_{k-2}).$ From Proposition \ref{prop:bano-ck} and the previous discussions, we see that the coefficient term is isomorphic to $\lambda^k h^1(C).$ Therefore we see that the polynomial of degree $2g$ is $(1+t)^{h^1(C)}.$ Moreover, we can also check that the coefficient of $t^k$ is zero when $k \geq 2g+1$ since 
$$ h(C_k)-(1+\LL)h(C_{k-1})+\LL h(C_{k-2}) $$
$$ = (\frac{1-\LL^{ k-g+1}}{1-\LL}-(1+\LL)\frac{1-\LL^{ k-g}}{1-\LL}+\LL\frac{1-\LL^{ k-g-1}}{1-\LL}) h(J(C))=0. $$
Therefore we see that $(1-t)(1-\LL t)Z(C,t)$ is indeed the polynomial $(1+t)^{h^1(C)}.$
\end{proof}

Especially, we have the following identity.

\begin{rema}
$$ Z(C,\LL^{ i}) = \frac{(1+\LL^{ i})^{h^1(C)}}{(1-\LL^{ i})(1-\LL^{ i+1})} $$
\end{rema}

We will use the above identity to compare results in \cite{Bano;thesis, Bano1, HL2}.

\bigskip

\section{A new proof for rank 2 cases}

When $r=2,$ S. del Ba\~no computed the motives of $M(2,L)$ in \cite{Bano2}. Using his formula, the second named author obtained the following formula in \cite{Lee}.

\begin{theo}
The motivic Poincar\'e polynomial of $M(2,L)$ has the following decomposition.
$$ \chi(M(2,L)) = \sum_{k=0}^{g-2} \chi(C_{k})  (\mathbb{L}^{ k} + \mathbb{L}^{ 3g-3-2k}) + \chi(C_{g-1})  \mathbb{L}^{ g-1}. $$
\end{theo}

Now let us give a completely new proof of the above theorem using the previous discussions.

\begin{proof}
From del Ba\~no's theorem (cf. Theorem \ref{Bano;inversion}), we see that the motivic Poincar\'e polynomial of $M(2,L)$ is as follows.
$$ \chi(M(2,L)) = Z(C,\LL) - \chi(J(C))  (\frac{\LL^{ g}}{(1-\LL)(1-\LL^{ 2})}) $$

On the other hand, we have the following identity from Proposition \ref{DecZeta-chow}.
$$ Z(C,\LL) = \sum_{k=0}^{g-2} \chi(C_{k})  (\mathbb{L}^{ k} + \mathbb{L}^{ 3g-3-2k}) + \chi(C_{g-1})  \mathbb{L}^{ g-1} + \chi(J(C))  (\frac{\LL^{ g}}{(1-\LL)(1-\LL^{ 2})}) $$

Therefore, we obtain the following desired result.
\end{proof}

\bigskip

\section{Rank 3 cases}

Now let us discuss rank 3 cases. Let $C$ be a smooth projective curve of genus $g$. Let
$M$ be the moduli space of vector bundles on $C$, stable of rank 3
and fixed determinant of degree 1.

\begin{theo}\label{rank3chow}
The motivic Poincar\'e polynomial of $M(3,L)$ has the following decomposition.
$$ \chi(M(3,L)) = \sum_{k_1+k_2 < 2(g-1)} \chi(C_{k_1} \times C_{k_2})  (\mathbb{L}^{ k_1+2k_2} + \mathbb{L}^{ 8g-8-2k_1-3k_2}) + $$
$$ \sum_{k_1+k_2=2(g-1), k_1<g-1} \chi(C_{k_1} \times C_{k_2})  (\mathbb{L}^{ k_1+2k_2} + \mathbb{L}^{ 8g-8-2k_1-3k_2}) + \chi(C_{g-1} \times C_{g-1})  \LL^{ 3(g-1)}. $$
\end{theo}
\begin{proof}
From Theorem \ref{Bano;ind-variety} and Proposition \ref{motiviczeta} we have the following identity.
$$ \chi(\mathrm{Bun}_{3,L}) = Z(C,\LL)  Z(C,\LL^{ 2}) = (\sum_{k=0}^{\infty} \chi(C_{k_1})  \mathbb{L}^{ k_1})  (\sum_{k=0}^{\infty} \chi(C_{k_2})  \mathbb{L}^{ 2k_2}) $$

And we have the following equality for $g \leq k \leq 2g-2.$
$$ \chi(C_k) = \chi(J(C))  (\sum_{l=0}^{k-g} \LL^{ l}) + \chi(C_{2g-2-k})  \LL^{ k-g+1} $$

Therefore we obtain the following identity.
$$ \chi(\mathrm{Bun}_{3,L}) = Z(C,\LL)  Z(C,\LL^{ 2}) = (\sum_{k=0}^{\infty} \chi(C_{k_1})  \mathbb{L}^{ k_1})  (\sum_{k=0}^{\infty} \chi(C_{k_2})  \mathbb{L}^{ 2k_2}) $$
$$ = ( \sum_{k_1=0}^{g-1} \chi(C_{k_1})  \mathbb{L}^{ k_1} + \sum_{k_1=g}^{2g-2} \chi(C_{2g-2-k_1})  \LL^{ 2k_1-g+1} + \sum_{k_1=g}^{\infty} \chi(J(C))  (\sum_{l=0}^{k_1-g} \LL^{ l})  \mathbb{L}^{ k_1} ) $$
$$  ( \sum_{k_2=0}^{g-1} \chi(C_{k_2})  \mathbb{L}^{ 2k_2} + \sum_{k_2=g}^{2g-2} \chi(C_{2g-2-k_2})  \LL^{ 3k_2-g+1} + \sum_{k_2=g}^{\infty} \chi(J(C))  (\sum_{l=0}^{k_2-g} \LL^{ l})  \mathbb{L}^{ 2k_2} ) $$
$$ = ( \sum_{k_1=0}^{g-1} \chi(C_{k_1})  \mathbb{L}^{ k_1} + \sum_{k_1'=0}^{g-2} \chi(C_{k_1'})  \LL^{ 3g-3-2k_1'} + \chi(J(C))  (\frac{\LL^{ g}}{(1-\LL)(1-\LL^{ 2})})) $$
$$  ( \sum_{k_2=0}^{g-1} \chi(C_{k_2})  \mathbb{L}^{ 2k_2} + \sum_{k_2'=0}^{g-2} \chi(C_{k_2'})  \LL^{ 5g-5-3k_2'} + \chi(J(C))  (\frac{\LL^{ 2g}}{(1-\LL^{ 2})(1+\LL+\LL^{ 2})})) $$

Let $\mathrm{Bun}_{3,L}^{\mathrm{un}}$ be the motive of unstable part and we can compute it using Harder-Narasimhan filtration. We have the following identity.
$$ \chi(\mathrm{Bun}_{3,L}) = \chi(M(3,L)) + \chi(\mathrm{Bun}_{3,L}^{\mathrm{un}}) $$

The motive of unstable part was computed by several authors, e.g. see \cite{Bano;thesis, DR, EK}. Using above expression, we obtain the following identity.
$$ \chi(M(3,L)) = \chi(\mathrm{Bun}_{3,L}) - \frac{\LL^{ 2g}}{1-\LL^{ 3}}  (\chi(J(C))  \chi(B\GG_m))  Z(C,\LL) $$
$$ - \frac{\LL^{ 2g-1}}{1-\LL^{ 3}}  (\chi(J(C))  \chi(B\GG_m))  Z(C,\LL) + \frac{\LL^{ 3g-1}}{(1-\LL^{ 2})^{ 2}}  (\chi(J(C))  \chi(B\GG_m))^{ 2} $$

Now let us do the calculation. We can check that the above formula can be reduced to the following formula.
$$ \chi(M(3,L)) = \chi(\mathrm{Bun}_{3,L}) - \frac{\LL^{ 2g-1}  (1+\LL)}{(1-\LL)(1-\LL^{ 3})}  \chi(J(C)  Z(C,\LL)  + \frac{\LL^{ 3g-1}}{(1-\LL)^{ 2}(1-\LL^{ 2})^{ 2}}  \chi(J(C))^{ 2} $$

For the terms containing $\chi(J(C))^2$ is as follows
$$ (\frac{\LL^{ 3g}}{(1-\LL)(1-\LL^{ 2})^2(1-\LL^{ 3})} - \frac{\LL^{ 3g-1}(1+\LL)^{ 2}}{(1-\LL)(1-\LL^{ 2})^{ 2}(1-\LL^{ 3})} + \frac{\LL^{ 3g-1}(1+\LL+\LL^{ 2})}{(1-\LL)(1-\LL^{ 2})^{ 2}(1-\LL^{ 3})})  \chi(J(C))^2 $$
and hence it vanishes. 

For the terms containing $\chi(J(C))$ is as follows
$$ \{ (\sum_{k_1=0}^{g-1} \chi(C_{k_1})  \mathbb{L}^{ k_1} + \sum_{k_1'=0}^{g-2} \chi(C_{k_1'})  \LL^{ 3g-3-2k_1'})  (\frac{\LL^{ 2g}}{(1-\LL^{ 2})(1-\LL^{ 3})}) $$
$$ + (\sum_{k_2=0}^{g-1} \chi(C_{k_2})  \mathbb{L}^{ 2k_2} + \sum_{k_2'=0}^{g-2} \chi(C_{k_2'})  \LL^{ 5g-5-3k_2'})  (\frac{\LL^{ g}}{(1-\LL)(1-\LL^2)}) $$
$$ - (\sum_{k=0}^{g-2} \chi(C_{k})  (\mathbb{L}^{ k} + \mathbb{L}^{ 3g-3-2k}) + \chi(C_{g-1})  \mathbb{L}^{ g-1})  (\frac{\LL^{ 2g-1}(1+\LL)^{ 2}}{(1-\LL^{ 2})(1-\LL^{ 3})}) \}  \chi(J(C)) $$
and hence it is equal to
$$ \{ \sum_{k=0}^{g-2} \chi(C_k)  \frac{\LL^{ 2k+g}(1-\LL^{ g-1-k})(1-\LL^{ 4g-4-4k})}{(1-\LL)(1-\LL^{ 2})} \}  \chi(J(C)). $$ 

Therefore it remains to show that 
$$ ( \sum_{k_1=0}^{g-1} \chi(C_{k_1})  \mathbb{L}^{ k_1} + \sum_{k_1'=0}^{g-2} \chi(C_{k_1'})  \LL^{ 3g-3-2k_1'} )( \sum_{k_2=0}^{g-1} \chi(C_{k_2})  \mathbb{L}^{ 2k_2} + \sum_{k_2'=0}^{g-2} \chi(C_{k_2'})  \LL^{ 5g-5-3k_2'} ) $$
$$ + \{ \sum_{k=0}^{g-2} \chi(C_k)  \frac{\LL^{ 2k+g}(1-\LL^{ g-1-k})(1-\LL^{ 4g-4-4k})}{(1-\LL)(1-\LL^{ 2})} \}  \chi(J(C)) $$ 
is equal to 
$$ \sum_{k_1+k_2 < 2(g-1)} \chi(C_{k_1} \times C_{k_2})  (\mathbb{L}^{ k_1+2k_2} + \mathbb{L}^{ 8g-8-2k_1-3k_2}) + $$
$$ \sum_{k_1+k_2=2(g-1), k_1<g-1} \chi(C_{k_1} \times C_{k_2})  (\mathbb{L}^{ k_1+2k_2} + \mathbb{L}^{ 8g-8-2k_1-3k_2}) + \chi(C_{g-1} \times C_{g-1})  \LL^{ 3(g-1)}. $$

Now let us decompose the last sum into
$$ \sum_{k_1 \leq g-1, k_2 \leq g-1, (k_1,k_2) \neq (g-1,g-1)} \chi(C_{k_1} \times C_{k_2})  (\mathbb{L}^{ k_1+2k_2} + \mathbb{L}^{ 8g-8-2k_1-3k_2})  + \chi(C_{g-1} \times C_{g-1})  \LL^{ 3(g-1)} $$
$$ + \sum_{k_1 < g-1, g \leq k_2 \leq 2(g-1)-k_1} \chi(C_{k_1} \times C_{k_2})  (\mathbb{L}^{ k_1+2k_2} + \mathbb{L}^{ 8g-8-2k_1-3k_2}) $$
$$ + \sum_{k_2 < g-1, g \leq k_1 < 2(g-1)-k_2} \chi(C_{k_1} \times C_{k_2})  (\mathbb{L}^{ k_1+2k_2} + \mathbb{L}^{ 8g-8-2k_1-3k_2}) $$

From the Proposition, we see that 
$$ \sum_{k_1 < g-1, g \leq k_2 \leq 2(g-1)-k_1} \chi(C_{k_1} \times C_{k_2})  (\mathbb{L}^{ k_1+2k_2} + \mathbb{L}^{ 8g-8-2k_1-3k_2}) $$
$$ = \sum_{k_1 < g-1, g \leq k_2 \leq 2(g-1)-k_1} \chi(C_{k_1})  (\chi(J(C))  (\sum_{l=0}^{k_2-g} \LL^{ l}) + \chi(C_{2g-2-k_2})  \LL^{ k_2-g+1})  (\mathbb{L}^{ k_1+2k_2} + \mathbb{L}^{ 8g-8-2k_1-3k_2}) $$

$$ = \sum_{k_1 < g-1, g \leq k_2 \leq 2(g-1)-k_1} \chi(C_{k_1})  \chi(J(C))  (\sum_{l=0}^{k_2-g} \LL^{ l})  (\mathbb{L}^{ k_1+2k_2} + \mathbb{L}^{ 8g-8-2k_1-3k_2}) + $$
$$ \sum_{k_1 < g-1, g \leq k_2 \leq 2(g-1)-k_1} \chi(C_{k_1})  \chi(C_{2g-2-k_2})  \LL^{ k_2-g+1}  (\mathbb{L}^{ k_1+2k_2} + \mathbb{L}^{ 8g-8-2k_1-3k_2}) $$
and 
$$ \sum_{k_2 < g-1, g \leq k_1 < 2(g-1)-k_2} \chi(C_{k_1} \times C_{k_2})  (\mathbb{L}^{ k_1+2k_2} + \mathbb{L}^{ 8g-8-2k_1-3k_2}) $$
$$ = \sum_{k_2 < g-1, g \leq k_1 < 2(g-1)-k_2} \chi(C_{k_2})  (\chi(J(C))  (\sum_{l=0}^{k_1-g} \LL^l) + \chi(C_{2g-2-k_1})  \LL^{ k_2-g+1})  (\mathbb{L}^{ k_1+2k_2} + \mathbb{L}^{ 8g-8-2k_1-3k_2}) $$

$$ = \sum_{k_2 < g-1, g \leq k_1 < 2(g-1)-k_2} \chi(C_{k_2})  \chi(J(C))  (\sum_{l=0}^{k_1-g} \LL^{ l})  (\mathbb{L}^{ k_1+2k_2} + \mathbb{L}^{ 8g-8-2k_1-3k_2}) + $$
$$ \sum_{k_2 < g-1, g \leq k_1 < 2(g-1)-k_2} \chi(C_{k_2})  \chi(C_{2g-2-k_1})  \LL^{ k_1-g+1}  (\mathbb{L}^{ k_1+2k_2} + \mathbb{L}^{ 8g-8-2k_1-3k_2}). $$ \\

Therefore we can check that the terms containing $\chi(J(C))$ is the same from the following identity
$$ \frac{(x^{k+2g}-x^{3g-3})(1-x^{2g-2-2k})}{(1-x)(1-x^2)} - \frac{(x^{k+2g+1}-x^{2g+k-2})(1-x^{3g-3-3k})}{(1-x)(1-x^3)} $$
$$ +\frac{(x^{2k+g}-x^{5g-4-2k})(1-x^{g-2-k})}{(1-x)^2} - \frac{(x^{2k+g+1}-x^{4g-k-2})(1-x^{2g-4-2k})}{(1-x)(1-x^2)} $$
$$ =\frac{x^{2k+g}(1-x^{g-1-k})(1-x^{4g-4-4k})}{(1-x)(1-x^2)} $$
where $x$ is a formal variable. Therefore we obtain the desired result.
\end{proof}

\bigskip

\section{Final remarks and discussions}

We think that the one can apply the above strategy to obtain new motivic decompositions for other cases. We leave this task for future investigations.

\subsection{Completion of Grothendieck group}

Behrend and Dhillon computed the class of $M(r,L)$ in the dimensional completion of the Grothendieck ring of varieties in \cite{BD}. Hoskins and Lehalleur studied Voevodsky's motives of moduli stacks of vector bundles on curves in \cite{HL1, HL2}. Using results in \cite{BD, HL1, HL2} we have the following result using the same proof in the previous sections. 

\begin{theo}\label{decompositionVoevodsky}
(1) Let $r=2, d=1.$ Then $[ M(2,L) ]$ is equal to the following motivic class in $\widehat{K_0}(\bVar)$ (or in $K_0(\widehat{\DM}_{\gm})$). \\
$$ \sum_{k=0}^{g-2} [ C_{k} ]  (\mathbb{L}^{ k} + \mathbb{L}^{ 3g-3-2k}) + [ C_{g-1} ]  \mathbb{L}^{ g-1}. $$
(2) Let $r=3, d=1.$ Then $[ M(3,L) ]$ is equal to the following motivic class in $\widehat{K_0}(\bVar).$
$$ \sum_{k_1+k_2 < 2(g-1)} [ C_{k_1} \times C_{k_2} ]  (\mathbb{L}^{ k_1+2k_2} + \mathbb{L}^{ 8g-8-2k_1-3k_2}) + $$
$$ \sum_{k_1+k_2=2(g-1), k_1<g-1} [ C_{k_1} \times C_{k_2} ]  (\mathbb{L}^{ k_1+2k_2} + \mathbb{L}^{ 8g-8-2k_1-3k_2}) + [ C_{g-1} \times C_{g-1} ]  \LL^{3(g-1)}. $$
(3) For $r=2,3,$ the class of $M(r,L)$ has decomposition of the same type as above in $K_0(\widehat{\DM}_{\gm}).$
\end{theo}
\begin{proof}
(1) From \cite{HL2} and the Proposition \ref{DecZeta-var}, we have the following identity.

$$ [ \mathrm{Bun}_{2,L} ] = \LL^3 Z(C,\LL) $$ 
$$ = \sum_{k=0}^{g-2} [ C_{k} ]  (\mathbb{L}^{ k} + \mathbb{L}^{ 3g-3-2k}) + [ C_{g-1} ]  \mathbb{L}^{ g-1} + [ J(C) ]  (\frac{\LL^{ g}}{(\LL-1)(\LL^{ 2}-1)}) $$

Let $\mathrm{Bun}_{2,L}^{\mathrm{un}}$ be the motive of unstable part and we can compute it using Harder-Narasimhan filtration. We have the following identity.
$$ [ \mathrm{Bun}_{2,L} ] = [ M(2,L) ] + [ \mathrm{Bun}_{2,L}^{\mathrm{un}} ] $$

Let $E$ be an unstable rank 2 bundle with determinant $L$ on $C.$ From the Harder-Narasimhan filtration we see that there is the following short exact sequence
$$ 0 \to L_1 \to E \to L_2 \to 0 $$
where $L_1, L_2$ two line bundles such that $L_1 \otimes L_2 = L$ and $\deg L_1 > \deg L_2.$ We may assume that $\deg L=1$ and then $\deg L_1$ can be $1,2, \cdots .$ For fixed $L_1, L_2$ then the extensions are parametrized by $\PP \Ext^1(L_2, L_1).$ Let $d$ be the degree of $L_1.$ From Riemann-Roch formula we see that $\dim \Hom(L_2, L_1) - \dim \Ext^1(L_2, L_1) = 1 + \deg(L_1) - \deg(L_2) - g=1+d-(1-d)-g=2d-g.$

Let us define a locally closed subset of $J(C)$ as follows. 
$$ B_i=\{L_1 \in \Pic^d(C) | \dim \Hom^0(L_2, L_1)= i \} $$ 

Then there is the following natural decomposition.
$$  \Pic^d(C) = B_0 \sqcup B_1 \sqcup \cdots $$

Let $L_1 \in B_i.$ The motivic class of the stack parametrizing the extensions of the form 
$$ 0 \to L_1 \to E \to L_2 \to 0 $$
is isomorphic to the following class.
$$ [B_i]  [[\Ext^1(L_2,L_1)]/[\Hom(L_2,L_1)]] = [B_i]  \frac{\LL^{ g}}{\LL^{ 2d}} $$

Therefore the motivic Poincar\'e polymonial of the unstable bundles is as follows.
$$ [\Bun_{2,L}^{\mathrm{un}}] = [J(C)]  \frac{\LL^{ g}}{(\LL-1)} \sum_{d=1}^{\infty}\frac{1}{\LL^{ 2d}} = [J(C)]  (\frac{\LL^{ g}}{(\LL-1)(\LL^{ 2}-1)}) $$ 

Therefore, we obtain the following desired result.
$$ [M(2,L)] = [ \mathrm{Bun}_{2,L} ] - [J(C)]  (\frac{\LL^{ g}}{(\LL-1)(\LL^{ 2}-1)}) $$
$$ = \sum_{k=0}^{g-2} [C_{k}]  (\mathbb{L}^{ k} + \mathbb{L}^{ 3g-3-2k}) + [C_{g-1}]  \mathbb{L}^{ g-1} $$

(2) The proof is similar to the above case. We can compute the motivic class of unstable bundles as above and do the same computation as in the proof of Theorem \ref{rank3chow}. \\

(3) By applying $\chi_c$ to $M(r,L)$ and the motivic classes in (1), (2), we obtain the conclusion.
\end{proof}

\begin{rema}
The motivic decompositions discussed so far should be reflected via realization functors. For example, the above motivic decompositions imply interesting decompositions of Hodge diamonds of moduli spaces. 
\end{rema}

\subsection{Derived categories}

As we mentioned, Orlov suggested that derived categories of coherent sheaves and motives of algebraic varieties will be closely related in \cite{Orlov}. We expect that there will be semiorthogonal decompositions of derived categories of coherent sheaves which are compatible with the motivic decompositions we discussed so far. For example, for rank 3 case we expect that there will be compatible semiorthogonal decompositions. See Conjecture \ref{SODrank2} and Conjecture \ref{SODrank3} for a precise statements.

\subsection{Fukaya categories}

Being motivated by homological mirror symmetry (cf. \cite{KKOY}), we also expect there will be a corresponding decompositions of Fukaya categories and quantum cohomology groups of the moduli spaces. See \cite{Lee} for rank two cases and references therein for more details and evidences. For rank 3 case, we have the following conjecture.

\begin{conj}
The Karoubian completion of Fukaya category of $M(3,L)$ will have the following orthogonal decomposition.
$$ Fuk^{\pi}(M(3,L)) = \langle \cdots, Fuk^{\pi}(C_{k_1} \times C_{k_2}), Fuk^{\pi}(C_{k_1} \times C_{k_2}), \cdots, Fuk^{\pi}(C_{g-1} \times C_{g-1}) \rangle $$
where 
$ (k_1,k_2)$ is a pair of nonnegative integers satisfying $k_1+k_2 < 2(g-1)$ or $k_1+k_2=2(g-1), k_1<g-1.$
\end{conj}

\section*{Acknowledgements}
 
It is a pleasure to express our deep gratitude to M. S. Narasimhan for drawing our attention to this problem and providing insightful suggestions about the decompositions. We also thank him for many helpful discussions which took place during several visits to the Indian Institute of Science (Bangalore). We thank Indian Institute of Science for wonderful working conditions and kind hospitality. We also thank Gadadhar Misra for kind hospitality during our stay in IISc. The second named author thanks Minhyong Kim, Valery Lunts and Jinhyun Park for helpful discussions and encouragements. Parts of this work was done when KSL was a Young Scientist Fellow of IBS-CGP and he was partially supported by IBS-R003-Y1. He also thanks Ludmil Katzarkov and Simons Foundation for partially supporting this work via Simons Investigator Award-HMS. TG is supported by Ministerio de Ciencia e Innovaci\'on of Spain (grants MTM2016-79400-P, PID2019-108936GB-C21, and ICMAT Severo Ochoa project SEV-2015-0554) and CSIC (2019AEP151 and \textit{Ayuda extraordinaria a Centros de Excelencia 
Severo Ochoa} 20205CEX001).

\bibliographystyle{amsplain}

\end{document}